\documentclass{amsart}

\usepackage[all]{xy}
\usepackage{ctex}
\usepackage{amsmath}
\usepackage{graphicx}
\usepackage{amssymb}
\usepackage{latexsym}
\usepackage{mathrsfs}
\newtheorem{theorem}{Theorem}[section]
\newtheorem{lemma}[theorem]{Lemma}

\theoremstyle{definition}
\newtheorem{definition}[theorem]{Definition}

\newtheorem{proposition}[theorem]{Proposition}
\newtheorem{corollary}[theorem]{Corollary}
\newtheorem{question}[theorem]{Question}

\theoremstyle{remark}
\newtheorem{remark}[theorem]{Remark}

\theoremstyle{notation}

\theoremstyle{claim}

\numberwithin{equation}{section}

\begin{document}

\title{On Schauder equivalence relations }

\author{Xin Ma}
\address{Department of Mathematics,
         Texas A\&M University,
         Collage Station, TX 77843}

\email{dongodel@math.tamu.edu}


\date{July 20, 2014.}



\begin{abstract}
In this paper, we study Schauder equivalence relations, which are Borel equivalence relations generated by Banach spaces with basic sequences. We prove that the set of equivalence relations generated by basic sequences has boundaries. In addition, we prove that both $\mathbb{R}^\mathbb{N}/l_p$ and $\mathbb{R}^\mathbb{N}/c_0$ are not reducible to the equivalence relation generated by Tsirelson space $T$ with the unit vector basis $\{t_n\}$. We also shows that Borel equivalence relations generated by $\alpha-$Tsirelson spaces are mutually incompatible. Based on this argument, we show that any basis of Schauder equivalence relations must be of cardinal $2^{\omega}$.

\end{abstract}

\maketitle

\section{Introduction}

The Borel reducibility hierarchy of equivalence relations on Polish spaces now becomes the main focus of invariant descriptive set theory, which has been an essential branch of the descriptive set theory. One of the most important kind of equivalence relations is the orbit equivalence relation generated by actions of Polish groups. A lot of important results and essential tools have been investigated. A separable Banach space with its norm topology can be regarded as a Polish abelian group. By a well-known theorem of Mazur (see Theorem 1.a.5 in \cite{L-T}), it admits a basic sequence. Then the subspace generated by such a sequence can be regarded as a Borel subgroup of $\mathbb{R}^\mathbb{N}$ ($\mathbb{N}=\{0,1,2\ldots\}$). Then, its natural action on $\mathbb{R}^\mathbb{N}$ generates a Borel equivalence relation. Some important situations have been studied thoroughly. For example, Dougherty and Hjorth proved that for any $p,q\in [1,\infty]$,
$$\mathbb{R}^\mathbb{N}/l_p\leq_B \mathbb{R}^\mathbb{N}/l_q \Longleftrightarrow p\leq q,$$
while $\mathbb{R}^\mathbb{N}/l_p$ and $\mathbb{R}^\mathbb{N}/c_0$ are Borel incomparable. ( see \cite{D-H} and \cite{H}).

Some kind of general form of the equivalence relations were further investigated successfully. Professor Ding introduced a kind of $l_p-$like equivalence relations $E((X_n);p)$. Let $(X_n, d_n),n<\omega$ is a sequence of pseudo-metric spaces with $p \geq1$,  He set that for any $x,y\in \Pi_{n < \omega}X_n$,
$$(x,y)\in E((X_n);p)\Longleftrightarrow \sum\limits_{n<\omega}d_n(x(n),y(n))<\infty,$$
He found that the reducibility between such equivalence relations are closely related to finitely H\"{o}lder(${\alpha}$) embedding by providing criteria of the reducibility. His theorem provides a lot of reducibility and non-reducibility results in the equivalence relations related to classical Banach spaces of the form $E((L_r);p)$ and $E((c_0);q)$. (see \cite{D2} and \cite{D1}).

Another kind of $l_p-$like equivalence relations $\mathrm{E}_f$ was introduced by M\'{a}trai. For any $f \colon [0,1]\rightarrow \mathbb{R}^+$ He considered the following relation on $[0,1]^\mathbb{N}$:
$$x \mathrm{E}_f y \Longleftrightarrow \sum\limits_{i<\omega}f(|x(n)-y(n)|)<\infty.$$
He embedded any a liner chain of the order $P(\omega)/\mathrm{Fin}$ into the set of the equivalence relations between $\mathbb{R}^\mathbb{N}/l_p$ and $\mathbb{R}^\mathbb{N}/l_q$, where $1 \leq p<q<\infty$, to answer a problem of Gao in negative(See \cite{M}).

More recently, Yin \cite{Y} has moved further to embed whole $P(\omega)/\mathrm{Fin}$ into the set of the equivalence relations between $\mathbb{R}^\mathbb{N}/l_p$ and $\mathbb{R}^\mathbb{N}/l_q$ to show the reducibility order of Borel equivalence relations between $\mathbb{R}^\mathbb{N}/l_p$ and $\mathbb{R}^\mathbb{N}/l_q$ are rather complex.

In this paper, we would like to study Schauder equivalence relations, which are Borel equivalence relations generated by Banach spaces with basic sequences. we firstly prove following two theorems in general. We denote $E(X,(x_n))$ the equivalence relation generated by Banach space $X$ with the basic sequence $\{x_n\}$. For some terminology in functional analysis, a \emph{subspace} $Y$ of a Banach space $X$ always means that $Y$ is a closed subspace in $X$.

\begin{theorem}
Let $Y$ be a Banach space, with $\{y_n\}$ being a Schauder basis. If $X$ is a subspace of $Y$, then there is a subspace $Z$ of $X$, with a basis $\{z_n\}$, such that $E(Z,(z_n)) \leq_B E(Y,(y_n))$.

Furthermore, If $X$ is a subspace of $Y$, with a normalized subsymmetric basic sequence $\{x_n\}$, then, $E(X,(x_n)) \leq_B E(Y,(y_n))$.
\end{theorem}

\begin{theorem}
Both $X$ and $Y$ are Banach spaces. $\{x_n\}$ and $\{y_n\}$ are normalized basic sequences of $X$ and $Y$, respectively.
If two following conditions hold:

{\rm (1)}\quad for all subsequences $\{x_{b_n}\}$ of $\{x_n\}$, $\{x_{b_n}\}$ does not dominate the $\{y_n\}$, and

{\rm (2)}\quad $\{y_n\}$ is lower semi-homogeneous,

then $E(X,(x_n)) \nleq_B E(Y,(y_n))$.
\end{theorem}

Using the theorems and their proof above, we can show $\mathbb{R}^\mathbb{N}/l_1$ and $\mathbb{R}^\mathbb{N}/c_0$ are minimal equivalence relations in the order of reducibility among all the equivalence relation of the form $E(X,(x_n))$. Of course, the minimality of $\mathbb{R}^\mathbb{N}/l_1$ can also be obtained by the property that $\mathbb{R}^\mathbb{N}/l_1$ is a minimal turbulent equivalence relation. Then a natural question, corresponding to a well-known theorem in Banach spaces, arises now: For all $E(X,(x_n))$, whether for some $p\geq 1$, $\mathbb{R}^\mathbb{N}/l_p \leq E(X,(x_n))$ or $\mathbb{R}^\mathbb{N}/c_0 \leq E(X,(x_n))$ always holds for one case. We answer this question in negative. In fact, just the equivalence relation generated by Tsirelson space $(T,(t_n))$ (the dual of the original Tsirelson space) witness a different situation.

Using the theorems above and the method addressed on turbulent ideals introduced by Farah(see \cite{F2} and \cite{F1}), we can prove following theorems.

\begin{proposition}
For any $p\geq 1$, Neither $\mathbb{R}^\mathbb{N}/l_p$ nor $\mathbb{R}^\mathbb{N}/c_0$ are Borel reducible to $E(T,(t_n))$.
\end{proposition}

We say any two equivalence relations $E$ and $F$ in a class $\mathcal{C}$ are compatible if there is an equivalence relation $R\in \mathcal{C}$ such that both $R\leq_B E$ and $R\leq_B F$ hold. We say subclass $\mathcal{B}$ of a class $\mathcal{C}$ of equivalence relations is a basis of $\mathcal{C}$ if for any $E\in \mathcal{C}$, there is a $F \in \mathcal{B}$ such that $F\leq_B E$. Using these terminologies, we can prove the following theorem and corollary .

Firstly, Corresponding to another theorem in Banach space theory that any $\alpha$-Tsirelson space and $\beta$-Tsirelson space are totally incomparable if $\alpha \neq \beta$, we have:
\begin{theorem}
$E(T_{\alpha},(t_n^{\alpha}))$ and $E(T_{\beta},(t_n^{\beta}))$ are incompatible among all Schauder equivalence relations.
\end{theorem}

Similar to the result of Farah(\cite{F2}), we have:
\begin{corollary}
Every basis of the class of Schauder equivalence relations must be of cardinal $2^{\omega}$
\end{corollary}

\section{Preliminaries}
In this section, we recall some basic notions concerning descriptive set theory, Banach spaces and Ideals. For the standard terminology in descriptive set theory we refer to \cite{G}, \cite{Ka} and  \cite{K}. We call a topological space \emph{Polish} if it is separable and completely metrizable. A \emph{Polish group} is a topological group with a compatible Polish topology. If $X$ is a Polish space and $G$ is a Polish group with an action $\cdot$ to $X$, then the orbit equivalence relation $E_G^X$ is defined by
$$xE_G^X y\iff \exists g\in G (g\cdot x=y).$$

Let $X$, $Y$ are Polish spaces and $E$, $F$ are equivalence relations on $X$ and $Y$, respectively. A \emph{Borel reduction} from $E$ to $F$ is a Borel function $\theta \colon X\rightarrow Y$ witnesses that
$$xEy \iff \theta (x)F\theta (y)$$
for all $x,y\in X$. In this case, we say that $E$ is \emph{Borel reducible} to $F$, denoted by $E \leq_B F$, if there is a Borel reduction from $E$ to $F$. We say $E$ and $F$ are \emph{Borel equivalent}, denoted by $E\thicksim_B F$, If $E \leq_B F$ and $F \leq_B E$. We call $E$ \emph{strictly Borel reducible} to $F$, denoted by $E<_B F$, if $E \leq_B F$ but not $F \leq_B E$. If $E$ is Borel reducible to a Borel countable equivalence relation, we say $E$ \emph{essentially countable}.

Hjorth once studied a dynamical property of group actions called \emph{turbulence} and proved that any equivalence relations generated by turbulent actions is not essentially countable. Related to turbulent equivalence relations, Hjorth(see \cite{H}) also proved the 5th dichotomy theorem as follows.

\begin{theorem}[\textup{Hjorth} \cite{H}]
Let a Borel equivalence relation $E \leq_B \mathbb{R}^\mathbb{N}/l_1$. Then exactly one of the following holds:

{\rm (1)}\quad $E$ is essentially countable;

{\rm (2)}\quad $\mathbb{R}^\mathbb{N}/l_1 \leq_B E$.
\end{theorem}

As $\mathbb{R}^\mathbb{N}/l_1$ is turbulent, we thus know that it is a minimal equivalence relation generated by a turbulent action.

Given a Banach space $X$, \emph{a Schauder basis} $\{x_n\}$ of $X$ means that any $x\in X$ can be expanded by the form $x= \sum\limits_{n=0}^{\infty}a_nx_n$ for a unique $\{a_n\}$ in $\mathbb{R}^\mathbb{N}$. A sequence $\{x_n\}$ is a basis of its closed linear span $[x_n]^{\infty}_{n=0}$ is called a \emph{basic sequence}. we say $\{x_n\}$ unconditional if $x= \sum\limits_{n=0}^{\infty}a_nx_n$ converges unconditionally. Here, a series $\sum\limits_{n=0}^{\infty}x_n$ converges unconditionally means that the series $\sum\limits_{n=0}^{\infty}x_{\pi(n)}$ converges for every permutation $\pi$ of the integers. A basic sequence $\{x_n\}$ is unconditional if and only if the convergence of $\sum\limits_{n=0}^{\infty}a_nx_n$ implies the convergence of $\sum\limits_{n=0}^{\infty}b_nx_n$ whenever $|b_n|\leq |a_n|$, for all $n$. For more details, please see see Proposition 1.c.6 in \cite{L-T}

Now, we mention the definition of the Schauder equivalence relations. This definition is due to Ding.

\begin{definition}[\textup{Ding} \cite{D3}]
For a basic sequence $\{x_n\}$ in Banach space $X$, we denote ${\rm coef}(X,(x_n))$ to be the set of all $a=(a_n)\in \mathbb{R}^\mathbb{N}$ such that $\sum a_nx_n$ converges. Define the Schauder equivalence relation $E(X,(x_n))$ by, for any $x,y\in \mathbb{R}^\mathbb{N}$,
$$(x,y)\in E(X,(x_n))\iff x-y\in {\rm coef}(X,(x_n)).$$
\end{definition}

It is worth noting that ${\rm coef}(X,(x_n))$  is a Borel subgroup of $\mathbb{R}^\mathbb{N}$. As the projection map on $X$ to each coordinates is continuous, we can easily check that ${\rm coef}(X,(x_n))$ is a Polishable subgroup of $\mathbb{R}^\mathbb{N}$.  $E(X,(x_n))$ is an orbit equivalence relation generated by the natural turbulent action of ${\rm coef}(X,(x_n))$ on $\mathbb{R}^\mathbb{N}$ and it is easy to check that such equivalence relations is Borel. We can easily see that $E(l_p, (e_n))$ is just the well-known equivalence relation $\mathbb{R}^\mathbb{N}/l_p$ and similarly, $E(c_0, (e_n))$ is $\mathbb{R}^\mathbb{N}/c_0$.

The followings is the definition of the block bases of a basic sequence.

\begin{definition}(\cite[Definition 1.a.10]{L-T})
\emph{Let $\{x_n\}$ be a basic sequence in a Banach space $X$. A sequence of non-zero vectors $\{u_j\}$ in $X$ of the form $u_j=\sum\limits^{p_{j+1}}_{p_j+1}a_nx_n$, with $\{a_n\}$ scalars and $p_1<p_2<\cdots$ an increasing sequence of integers, is called a \emph{block basic sequence} or briefly a \emph{block basis} of $\{x_n\}$.}
\end{definition}

 For two basic sequence $\{x_n\}$ and $\{y_n\}$ in $X$ and $Y$, if ${\rm coef}(X,(x_n))\subset {\rm coef}(Y,(y_n))$, we say that $\{x_n\}$ \emph{dominates} $\{y_n\}$, denoted by $\{x_n\}\gg\{y_n\}$ (see \cite{C-S}). If ${\rm coef}(X,(x_n))={\rm coef}(Y,(y_n))$, we say $\{x_n\}$ and $\{y_n\}$ are \emph{equivalent}. We call $\{x_n\}$ \emph{subsequence equivalent} if for any subsequence $\{x_{k_n}\}$ of $\{x_n\}$, ${\rm coef}(X,(x_n))={\rm coef}(X,(x_{k_n}))$. An unconditional subsequence equivalent basic sequence is called \emph{subsymmetric}.(see Definition 3.a.2 in \cite{L-T}). Furthermore, we call $\{x_n\}$ \emph{symmetric} if for any permutation $\pi$ of $\mathbb{N}$, ${\rm coef}(X,(x_n))={\rm coef}(X,(x_{\pi(n)}))$.

Another property we mention is called \emph{lower semi-homogeneous}. It means that any normalized block bases of $\{x_n\}$ dominates $\{x_n\}$, where $\{x_n\}$ is a normalized basis in $X$. To my knowledge, This property was firstly studied by Casazza and Bor-Luh Lin (see \cite{C-L}). Undoubtedly, $c_0$ and $l_p$, $1 \leq p< \infty$ are lower semi-homogeneous (see Theorem 2.a.9 in \cite{L-T}) but it is not true that only they have the property. In \cite{C-L}, there is such an example concerning a Lorentz sequence space. It is easy to check that, or by Proposition 1 in \cite{C-L}, every lower semi-homogeneous basis is an unconditional basis.

For a Tsirelson's space $T$ in this paper, we mean the dual space of the oringinal space constructed by Tsirelson. We would like to provide its definition here. In fact, we would like to move further to provide the definition of a general version of Tsirelson space, $T_{\alpha}$, here. Based on this definition, Tsirelson space, $T$ is just $T_{\alpha}$, when $\alpha$ is $1/2$. For any finite non-void subset $E$, $F$ of $\omega$, we denote $E<F$ for $\max(E)<\min(F)$, with $n<E$, instead of $\{n\}<E$, and with analogous meanings for $E\leq F$. For the space $c_{00}$, we mean the sequence space of all sequences of scalars which are eventually zero.

\begin{definition}(\cite[Example 2.e.1]{L-T})
For any $\alpha\in (0,1)$, define a sequence of norms $\|\cdot\|_m$ upon $c_{00}$ as follows:
fixing $x=\sum\limits_n a_nx_n\in c_{00}$, let $\|x\|_0= \max\limits_n|a_n|$. Then by induction, for $m\geq 0$
$$\|x\|_{m+1}=\max \{\|x\|_m,\alpha\cdot \max [\sum\limits_{j=1}^k\|E_jx\|_m]\},$$
where ``inner'' max is taken over all choices of finite subsets $\{E_j\}_{j=1}^k$ of $N$ as $k$ varies and such that  $k\leq E_1 <E_2<\ldots <E_k$.

Then $\|x\|=\lim \|x\|_m$ is a norm on $c_{00}$. The general Tsirelson space $T_{\alpha}$ is the completion of $c_{00}$ with the norm $\|\cdot\|$.
\end{definition}

When $\alpha$ is $1/2$, $T_{1/2}$ is the Tsirelson space $T$. This definition, as far as I know, is firstly introduced by Figiel and Johnson \cite{F-J} and further being studied by Cassazza, Johnson and Tzafriri and Shura(\cite{C-J-T}, \cite{C-S}). It is  well-known that $T_{\alpha}$ are spaces which contain no subspace isomorphic to any $l_p$ for $p\geq 1$ and $c_0$. They have similar properties in some way but can be totally different from each other. That is, for different $\alpha, \beta$, $T_{\alpha}$ and $T_{\beta}$ are \emph{totally incomparable} Banach spaces, i.e. they do not have same infinite-dimensional subspaces in the isomorphic view. For more details, see Definition 2.c.1 in \cite{L-T} and theorem X.a.3 in \cite{C-S}.

More ``Tsirelson-like'' spaces $T_h$ and $T_{\alpha,h}$ can be defined in the similar manner. If we define $\|x\|_{m+1}= \max \{\|x\|_m,\alpha \cdot \max [\sum\limits_{j=1}^{h(k)}\|E_jx\|_m]\}$ for some strictly increasing function $h$ from $\mathbb{N}$ to $\mathbb{N}$, we will obtain the $T_{\alpha,h}$. Similarly, when $\alpha$ is $1/2$, we obtain the $T_h$. It is worth noting that the basis $\{t_n^h\}$ in $T_h$ is equivalent to the basic sequence $\{t_{h(n)}\}$ of the basis $\{t_n\}$ in $T$. For more details, Please see \cite{C-J-T} and \cite{C-S}.

There are more sequence spaces which are generalization of $l_p$. We firstly mention Orlicz sequence spaces, which were firstly introduced by Orlicz. For more detail, Please see \cite{L-T}

\begin{definition}(\cite[Definition 4.a.1]{L-T})
An \emph{Orlicz Function} $M$ is a continuous non-decreasing and convex function defined for $t\geq 0$ such that $M(0)=0$ and $\lim\limits_{t\rightarrow \infty}M(t)= \infty$.
\end{definition}

For any Orlicz function $M$ we can define Banach space $l_M$ contains all sequences of scalar $x\in \mathbb{R}^\mathbb{N}$ such that $\sum\limits_{n=0}^{\infty}M(x(n)/\rho)<\infty$ for some $\rho$. On $l_M$, we can define a compatible norm as follows:
$$\|x\|=\inf\{\rho\ \colon \sum\limits_{n=0}^{\infty}M(x(n)/\rho)<\infty\}$$
The subspace $h_M$ of $l_M$, which contains all sequences $x\in l_M$ such that $\sum\limits_{n=0}^{\infty}M(x(n)/\rho)<\infty$ for all $\rho$, is of particular interest. It can be checked that $h_M$ is a closed subspace of $l_M$ and unit vectors $\{e_n\}$ form a symmetric basis of $h_M$. (see Proposition 4.a.2 in \cite{L-T}). The following property is called $\Delta^{'}-$condition.

\begin{definition}
An Orlicz function $M$ is said to satisfy the $\Delta^{'}-$\emph{condition} at zero if there is real numbers $c$ and $x_0$ such that for all $x,y\in [0,x_0]$, $M(xy)\geq cM(x)M(y)$.
\end{definition}

Now, we mention Lorentz sequence spaces $d(w,p)$, which were introduced firstly by Lorentz as function spaces. For more details, please also see \cite{L-T}.

\begin{definition}(\cite[Definition 4.e.1]{L-T})
Let $1\leq p< \infty$, and let $w=\{w_n\}$ be a non-decreasing sequence of positive numbers such that $w_0=1$, $\lim\limits_{n\rightarrow \infty} w_n=0$ and $\sum\limits_{n=0}^{\infty}w_n=\infty$. The Banach space of all sequence $x\in \mathbb{R}^\mathbb{N}$ for which
$$\|x\|=\sup\limits_{\pi}(\sum\limits_{n=1}^{\infty}\|x(\pi(n))\|^pw_n)^{1/p}<\infty,$$
where $\pi$ ranges over all the permutations of the integers, is denoted by $d(w,p)$ and it is called a \emph{Lorentz sequence space}
\end{definition}

For an ideal $\mathscr{I}$, we mean a set $\mathscr{I}\subset P(\omega)$ such that for any $A,B\in \mathscr{I}$, $A\cup B\in \mathscr{I}$ and if $C\subset A\in \mathscr{I}$, then $C\in \mathscr{I}$.  Such an ideal can be regarded as a subset of Cantor space $2^\mathbb{N}$ with the usual product topology. A Borel ideal thus means that the ideal is a Borel subset of $2^\mathbb{N}$. In this way, any Borel ideal is a Borel subgroups of $2^\mathbb{N}$ under the operation $\Delta$, where $x\Delta y=(x-y)\bigcup (y-x)$.  Then, the natural action, $\Delta$, of a Borel ideal $\mathscr{I}$ on $2^\mathbb{N}$ can generate a equivalence relation.

If such an action is turbulent, we say $\mathscr{I}$ is turbulent. In addition, we say an ideal $\mathscr{I}$ Polishable if it is a Polishable subgroup in $2^\mathbb{N}$.

A typical way to define an ideal is to use submeasures. A \emph{submeasure} on a set $A$ is any map $\phi \colon P(A)\rightarrow [0.\infty]$, satisfying $\phi(\emptyset)=0$, $\phi(\{a\})<\infty$ for all $a$, and $\phi(x)\leq \phi(x\bigcup y)\leq \phi(x)+\phi(y)$. A submeasure $\phi$ on $\mathbb{N}$ is \emph{lower-semicontinuous}, or LSC for brevity, if we have $\phi(x)=\sup_n\phi(x\bigcap[0,n))$ for all $x\in P(\mathbb{N})$. For any submeasure $\phi$, Define the \emph{tail submeasure} $\phi_{\infty}(x)=\inf_n(\phi(x\bigcap [n,\infty)))$. Now, following ideals will be considered.
$$\verb"Fin"_{\phi}=\{x\in P(\mathbb{N})\colon\ \phi(x)< \infty \}$$
$$\verb"Null"_{\phi}=\{x\in P(\mathbb{N})\colon\ \phi(x)=0 \}$$
$$\verb"Exh"_{\phi}=\{x\in P(\mathbb{N})\colon\ \phi_{\infty}(x)=0 \}$$

Using these terminology, a characterization theorem which is due to Solecki, can be arrived as follows:

\begin{theorem}(\cite[Theorem 3.5.1]{Ka})
Suppose that $\mathscr{I}\subset P(\omega)$ is an ideal, then following conditions are equivalent:

{\rm (1)}\quad $\mathscr{I}$ has the form $\verb"Exh"_{\phi}$, where $\phi$ is a LSC submeasure on $\mathbb{N}$.

{\rm (2)}\quad $\mathscr{I}$ is Polishable.

{\rm (3)}\quad $\mathscr{I}$ is an analytic P-ideal.

Furthermore, $\mathscr{I}$ is an $F_{\sigma}$ P-ideal iff $\mathscr{I}=\verb"Fin"_{\phi}=\verb"Exh"_{\phi}$, for some LSC submeasure.

\end{theorem}

\begin{remark}
We know that any Polishable ideal $\mathscr{I}$ has the form $\verb"Exh"_{\phi}$, where $\phi$ is a LSC submeasure, on $\mathbb{N}$ is turbulent if and only if $\phi(\{n\})\rightarrow 0$. \emph{See \cite{K1}}.
\end{remark}

A famous type of turbulent analytical P-ideals is the summable ideals. Here, we only mention $\mathscr{I}_{1/n}= \{A \colon \sum\limits_{n\in A} 1/n< \infty\}$. It is well-known that $E_{\mathscr{I}_{1/n}}\thicksim \mathbb{R}^\mathbb{N}/l_1$. For more details about ideals, we refer to \cite{Ka}. Given a Banach space $X$, and an unconditional basic sequence $\{x_n\}$ of $X$ such that $\sum x_n$ diveges. we can define an ideal as follows:
$$\mathscr{I}=\{A \colon \sum\limits_{n\in A} x_n\ \textrm{converges}\}.$$

In this manner, Farah defined  a kind of $\alpha-$Tsirelson ideals $ \mathscr{T}_{f,h,\alpha}$.(see \cite{F2} and \cite{F1}). Actually, by induction, He defined a LSC submeasure $\tau_{f,h,\alpha}$, which is similar to the definition of norm in Tsirelson space  to induce the ideals. In this paper, we do not need to deal with these submeasures. Thus, for more details about them, please see  \cite{F2} and \cite{F1}.
\section{Reducibility and non-reducibility}

In this section, we will mainly prove Theorem 1.1 and Theorem 1.2. We begin with the reducibility theorem. The following lemma is trivial but fundamental.

\begin{lemma}
Suppose that $\{x_n\}$ is a basic sequence in $X$ and $\{u_j=\sum\limits^{p_{j+1}}_{p_j+1}a_nx_n\}$, with $\{a_n\}$ scalars and $p_1<p_2<\cdots$ an increasing sequence of integers, is a block basis of $\{x_n\}$, Then $E(X,(u_n))\leq_B E(X,(x_n))$. In particular, for any subsequence $\{x_{k_n}\}$ of $\{x_n\}$, we have $E(X,(x_{k_n}))\leq_B E(X,(x_n))$.
\end{lemma}
\begin{proof}
The needed reduction $\theta$ from $\mathbb{R}^\mathbb{N}$ to $\mathbb{R}^\mathbb{N}$ can be easily constructed as follows. For any $c\in \mathbb{R}^\mathbb{N}$:
$$\theta(c)(n)=c_j\cdot a_n\  \textrm{if}\  p_j<n\leq p_{j+1}.$$
\end{proof}

Then, more propositions about Banach space are needed. for these propositions we refer to \cite{L-T}

\begin{proposition}(\cite[Proposition 1.a.11]{L-T})
Let $X$ be a Banach space with a Schauder basis $\{x_n\}$. Let $Y$ be a infinite dimensional subspace of $X$. Then there is a subspace $Z$ of $Y$ which has a basis which is equivalent to a block basis of $\{x_n\}$.
\end{proposition}

Using the lemma and the proposition above, we can easily prove the first part of Theorem 1.1, which is the case that if $X$ is a subspace of $Y$. However, this argument only asserts the ``existence'' of a needed equivalence relation, which cannot be satisfied to handle. On the other hand, in some special case, like $l_1$, we can show that if $X$ contains $l_1$ as its closed subspaces, then $\mathbb{R}^\mathbb{N}/l_1 \leq_B E(X,(x_n))$. One way to prove it is to use a well-known result of James.

\begin{theorem}[\textup{James} \cite{J}]
If a normed liner space contains a subspace isomorphic to $l_1$, then, for any positive number $\delta$, there is a sequence $\{u_i\}$ of members of the unit ball such that
$$(1-\delta)\cdot\sum |a_i|\leq \|\sum a_iu_i\|\leq \sum |a_i|$$
for all sequence of numbers $\{a_i\}$.
\end{theorem}

In fact, on the condition that $X$ is a Banach space, with a basis $\{x_n\}$, we can take $\{u_n\}$ a normalized block basis of $\{x_n\}$ (See the proof of Proposition 2.e.3 in \cite{L-T}). It implies that $\{e_n\}$ in $l_1$ is equivalent to a normalized block basis of $\{x_n\}$ in $X$. Now, using the lemma above, we have $\mathbb{R}^\mathbb{N}/l_1 \leq_B E(X,(x_n))$.

The next one is known as the Bessaga-Pelczynski selection principle.

\begin{proposition}(\cite[Proposition 1.a.12]{L-T})
Let $\{x_n\}$ be a Schauder basis of a Banach space $X$. Let $y_k=\sum\limits_{n=0}^{\infty}a_{n,k}x_n$, $k=1,2\ldots$, be a sequence of vectors such that:

{\rm (1)}\quad $\limsup\limits_k\|y_k\|>0$,

{\rm (2)}\quad $\lim\limits_ka_{n,k}=0$.

Then there is a subsequence $\{y_{k_j}\}$ of $\{y_k\}$ such that it is equivalent to a block basis of $\{x_n\}$
\end{proposition}

Using Bessaga-Pelczynski selection principle, we can say more when $X$ is a subspace of $Y$. If a Banach space $X$, with a normalize basis $\{x_n\}$ which weakly converges to $0$, is a subspace of $Y$ with a normalized basis $\{y_n\}$, then, for any $k$, there is a sequence of scalars $\{a_{n,k}\}$ such that $x_k=\sum\limits_{n=0}^{\infty}a_{n,k}y_n$. By Proposition 3.4, there is a subsequence $\{x_{k_n}\}$ of $\{x_n\}$ which is equivalent to a block basis $\{u_n\}$ of $\{y_n\}$. Thus, we have $E(X,(x_{k_n}))\leq_B E(Y,(y_n))$. That is $E([x_{k_n}],(x_{k_n}))\leq_B E(Y,(y_n))$.

Combining these arguments, we are ready to complete the proof of Theorem 1.1 as follows:

\begin{proof}
(Theorem 1.1) We can assume that $\{x_n\}$ is a normalized unconditional basis of $X$, is a subspace of $Y$. When $\{x_n\}$ is subsymmetric, It is well known that either $\{x_n\}$ is equivalent to the unit vector basis $\{e_n\}$ of $l_1$, or $\{x_n\}$ weakly converges to $0$ (see \cite{C-S}). Thus, no matter which case happens, $E(X,(x_n)) \leq_B E(Y,(y_n))$.
\end{proof}

For any Banach spaces, using Theorem 1.1, we can obtain following corollaries.

\begin{corollary}\label{3.3}
Let $X$ be a Banach space, which admits a normalized subsymmetric basis $\{x_n\}$. Then  $E(X,(x_n))$ is a minimal element, in the order of the $\leq_B$, of the set $\{E(X,(y_n))\colon \{y_n\} \textrm{ is a basis of } X\}$.
\end{corollary}

The corollary above actually implies that any two Schauder equivalence relations generated by different subsymmetric bases are Borel equivalent to each other. On the other hand, if we consider all Schauder equivalence relations generated by all basic sequences in a Banach space $X$, the corollary above is wrong. We can see counterexamples in Corollary 3.15.

Let $X$ be a Banach space with a Schauder basis $\{x_n\}$. For any $n\in \omega$, the linear functional $x_n^\ast$ on $X$ is defined by $x^\ast_n(\sum\limits_{i=0}^{\infty}a_ix_i)=a_n$ is a bounded linear functionals. Actually, $\|x^\ast_n\|\leq 2K/\|x_n\|$ where $K$ is the basis constant of $\{x_n\}$. We call $\{x_n\}$ \emph{shrinking} if $\{x^\ast_n\}$ form a Schauder basis of $X^\ast$ (see Proposition 1.b.1 in \cite{L-T} ). For another corollary, we need following theorem due to James.

\begin{theorem}(\cite[Theorem 1.c.9]{L-T})
Let $X$ be a Banach space with an unconditional basis $\{x_n\}$. Then $\{x_n\}$ is shrinking if and only if $X$ does not have a subspace isomorphic to $l_1$
\end{theorem}

It can be easily checked that any shrinking basis weakly converge to $0$. For any $X$, we denote $\mathcal{A}_X$ for the class, which contains all equivalence relations of the form $E(X,(x_n))$, where $\{x_n\}$ is a basic sequence in $X$. Then we arrive following corollary now.

\begin{corollary}
Let $X$ be a Banach space having unconditional bases, then for any two unconditional bases $\{x_n\}$ and $\{y_n\}$, $E(X,(x_n))$ and $E(X,(y_n))$ are compatible in $\mathcal{A}_X$.
\end{corollary}
\begin{proof}
If $X$ contains a copy of $l_1$, Then both $\mathbb{R}^\mathbb{N}/l_1\leq_B E(X,(x_n))$ and $\mathbb{R}^\mathbb{N}/l_1\leq_B E(X,(y_n))$ hold. If not, $\{x_n\}$ is shrinking and thus there is a subsequence $\{x_{k_n}\}$ of $\{x_n\}$ which is equivalent to a block basis $\{u_n\}$ of $\{y_n\}$. Thus we have $E(X,(x_{k_n}))\leq_B E(X,(y_n))$. Together with $E(X,(x_{k_n}))\leq_B E(X,(x_n))$, the conclusion is arrived.
\end{proof}

Based on the corollary above, a question arises naturally. The following one is asked by Liu:

\begin{question}[Rui Liu]
Whether there is a Banach space $X$ with two unconditional bases $\{x_n\}$ and $\{y_n\}$ such that $E(X,(x_n))$ and $E(X,(y_n))$ are not Borel equivalent?
\end{question}

For a kind of natural generalization of Theorem 1.1 for the case that $X$ is a subspace of $Y$, one may want to see what will happen if $\{x_n\}$ is not necessarily unconditional. Here we mention a famous result of Rosenthal. It may be helpful to clarify the situation to some extent. We say a sequence $\{x_n\}$ \emph{weak Cauchy} if for any function $x^*\in X^*$, we have $\lim\limits_{n\rightarrow \infty} x^*(x_n)$ exists.(see \cite{R})

\begin{theorem}[\textup{Rosenthal} \cite{R}]
Let $\{x_n\}$ be a bounded sequence in a Banach space $X$. Then, $\{x_n\}$ has a subsequence $\{x_{n_i}\}$ satisfying one of the two mutually exclusive alternatives:

{\rm (1)}\quad $\{x_{n_i}\}$ is equivalent to the unit vector basis of $l_1$.

{\rm (2)}\quad $\{x_{n_i}\}$ is a weak Cauchy sequence.
\end{theorem}

Now, we work in the case that $X$ is a subspace of $Y$, with a conditional normalized basis $\{x_n\}$ in $X$. It is clear that the proof of Theorem 1.1 highly depend on Proposition 3.4, which deals with the situation that $\{x_n\}$ weakly converges to $0$. In the meantime, James's result allows us to handle the space $l_1$. By applying theorem 3.9 we can pass $\{x_n\}$ to one of its subsequences, and the last case that we need to find out is when $\{x_{n_i}\}$ is a weak Cauchy sequence but not weakly converges to $0$ (i.e. \emph{non-trivial weak Cauchy}, see \cite{Ro}). The typical example is $c$ with its summing basis. Thus if one answer the following question, Theorem 1.1 in the case of conditional basis will be completed easily.

\begin{question}
Let $Y$ be a Banach space, with $\{y_n\}$ being a Schauder basis. If $X$, with a normalized non-trivial weak Cauchy basic sequence $\{x_n\}$, is a subspace in $Y$, Whether there is a subsequence $\{x_{b_n}\}$ of $\{x_n\}$, such that $E(X,(x_{b_n})) \leq_B E(Y,(y_n))$?
\end{question}

Now, we address Theorem 1.2. One of a standard approach to address this kind of theorem is involved. This approach was firstly used by Louveau and Velickovic (see \cite{L-V}) and developed by Dougherty and Hjorth (see \cite{D-H}). Given a reduction $\theta$ from $E$ to $F$, one can reorganize it to obtain another reduction $\theta'$ which is not only continuous but ``modular". It means that the sequence in the range of $\theta'$ are built by finite blocks, each of which depends on only one coordinate of the argument to the function. In this case, we call that $\theta'$ witness that $E\leq_A F$.  Before we give the proof of Theorem 1.2, we would like to provide a definition, which is initiated in \cite{D-H}.

\begin{definition}(\cite[Definition 2.1]{D-H})
\emph{Let $\vec{\epsilon}=(\epsilon_i)_{i<\omega}$, let $\mathbb{Z}(\vec{\epsilon})$ denote the set of all $x\in \mathbb{R}^\mathbb{N}$ such that $x(n)$ is an integer multiple of $\epsilon_n$ for all $n\in \omega$}.
\end{definition}

We can easily see that $\mathbb{Z}(\vec{\epsilon})$ and $\mathbb{Z}(\vec{\epsilon})\cap[-1,1]^{\mathbb{N}}$ are both Polish spaces. Now we are ready to prove Theorem 1.2.

\begin{proof}
(Theorem 1.2) we also assume that both $\{x_n\}$ and $\{y_n\}$ are normalized basis and just need to check there is no such reduction $\theta$ from $\mathbb{Z}(\vec{\epsilon})\cap[-1,1]^{\mathbb{N}}\rightarrow \mathbb{R}^\mathbb{N}$. Here, the value of $\vec{\epsilon}$ can be interpreted by $\epsilon_i=2^{-i}$.

The following steps (claim 1-4) is modified from the proof of \cite{D-H}, Theorem 2.2 claim (1-4).

\emph{\textbf{Claim 1}.}
$\forall j,k\in \mathbb{N}$, $\exists l\in \mathbb{N}$ and $s^{\ast}$ with $s^{\ast}(i)=m\epsilon_{k+i}$ for some $m\in \mathbb{Z}$ such that $|m\epsilon_{k+i}|\leq1$ for all $i<\textrm{len}(s^{\ast})$ and a comeager set $D\subseteqq \mathbb{Z}(\vec{\epsilon})\cap[-1,1]^{\mathbb{N}}$ s.t. for all $u,\hat{u}\in D$, if $u=rs^{\ast}w$ and $\hat{u}=\hat{r}s^{\ast}w$ for some $r$, $\hat{r}\in \mathbb{R}^{k}$ and $y\in \mathbb{R}^\mathbb{N}$, then: $\|\sum\limits_{i=l+1}^{\infty}(\theta(u)(i)-\theta(\hat{u})(i))y_i\|<2^{-j}$.

\begin{proof}
(Claim 1.) For each $l\in \mathbb{N}$, define $F_l$ from $\mathbb{Z}(\vec{\epsilon})\cap[-1,1]^{\mathbb{N}}$ to $\mathbb{R}^\mathbb{N}$ by:
$$F_l(u)=\max_{z,\hat{z}}\|\sum_{i=l+1}^{\infty}(\theta(z)(i)-\theta(\hat{z})(i))y_i\|,$$
where $z,\hat{z}$ are elements of $\mathbb{Z}(\vec{\epsilon})\cap[-1,1]^{\mathbb{N}}$ s.t. $\forall i\geq k$ $z(i)=\hat{z}(i)$. since $z-\hat{z}\in {\rm coef}(X,(x_n))$, $\theta(z)-\theta(\hat{z})\in {\rm coef}(Y,(y_n))$.  Thus,$\forall\epsilon$ $\exists N$ $\forall n>N$ $\|\sum\limits_{i=l+1}^{\infty}(\theta(z)(i)-\theta(\hat{z})(i))y_i\|<\epsilon$. Thus,  $\lim\limits_{l\rightarrow \infty}\|\sum\limits_{i=l+1}^{\infty}(\theta(z)(i)-\theta(\hat{z})(i))y_i\|=0$ and $\lim\limits_{l\rightarrow \infty}F_l(u)=0$ for all $u$. Therefore, for
 $\forall j$, we have $\forall u$ $\exists l$ $F_l(u)<2^{-j}$. If we denote $K_l=\{u \colon F_l(u)<2^{-j}\}$, $\bigcup K_l=\mathbb{Z}(\vec{\epsilon})\cap[-1,1]^{\mathbb{N}}$. Thus, there is a $l$ such that $K_l$ is not meager and then there are some finite sequence $t$ s.t. $N_t\Vdash K_l$. We can also extends $t$ s.t. $\textrm{len}(t)\geq k$. Take $t=r^{\ast}s^{\ast}$ where $\textrm{len}(r^{\ast})=k$. Consider $f\colon N_{r^{\ast}}\rightarrow N_r$ where $\textrm{len}(r^{\ast})=\textrm{len}(r)=k$. by $f(u)=\hat{u}$ s.t. $u(i)=\hat{u}(i)$ for $i\geq\textrm{len}(r)=k$.
 By the definition of $F_l(u)$, $u\in K_l$ iff $f(u)\in K_l$. and then we have $\forall r$ s.t. $\textrm{len}(r)=k$, $N_{rs^{\ast}}\Vdash K_l$. At last, take $D=\mathbb{Z}(\vec{\epsilon})\cap[-1,1]^{\mathbb{N}}-\bigcup_r(N_{rs^{\ast}}-K_l)$.
\end{proof}

We fix a dense $G_\delta$ set $C\subseteqq \mathbb{Z}(\vec{\epsilon})\cap[-1,1]^{\mathbb{N}}$ on which $\theta$ is continuous.

\emph{\textbf{Claim 2}.}
$\forall j,k,l\in \mathbb{N}$, $\exists s^{\ast\ast}$ with $s^{\ast\ast}(i)=m\epsilon_{k+i}$ for some $m\in \mathbb{Z}$ such that $|m\epsilon_{k+i}|\leq1$ for all $i<\textrm{len}(s^{\ast\ast})$ s.t. $\forall u,\hat{u}\in C$, if $x=rs^{\ast\ast}w$ and $\hat{x}=rs^{\ast\ast}\hat{w}$ for some $r\in \mathbb{R}^k$ and $w,\hat{w}\in \mathbb{R}^\mathbb{N}$, then $\|\sum\limits_{i=0}^{l}(\theta(u)(i)-\theta(\hat{u})(i))y_i\|<2^{-j}$

Furthermore, if $G$ is dense open in $\mathbb{Z}(\vec{\epsilon})\cap[-1,1]^{\mathbb{N}}$, then $s^{\ast\ast}$ can be chosen s.t. $N_{rs^{\ast\ast}}\subseteqq G$ for all $r\in \mathbb{R}^k$ s.t. $r(i)=m\epsilon_{i}$ for some $m$ such that $|m\epsilon_{i}|\leq 1$

\begin{proof}
(Claim 2.) we enumerate such $r$ by $r_0, r_1..... r_V$. By induction, define $t_0=\varnothing$. suppose we have $t_m$ as $C$ is comeager, there is a $z\in N_{r_mt_m}\cap C$. since $\theta$ is continuous on $C$, there is a neighborhood $O$ of $z$ s.t. for $\forall x,\hat{x}\in C\cap O$, $\sum\limits_{i=0}^{l}|\theta(u)(i)-\theta(\hat{u})(i)|<2^{-j}$. Thus, $\|\sum\limits_{i=0}^{l}(\theta(u)(i)-\theta(\hat{u})(i))y_i\|<\sum\limits_{i=0}^{l}|\theta(u)(i)-\theta(\hat{u})(i)|<2^{-j}$. this $O$ can be taken as $N_{r_m\tilde{t}_m}$ s.t. $t_m\subseteqq\tilde{t}_m$, and we can extends $\tilde{t}_m$ to be $t_{m+1}$ such that $N_{r_mt_{m+1}}\subseteqq G$ as $G$ is dense open. At last take $s^{\ast\ast}=t_V$.

 Then $x=r_ms^{\ast\ast}y$ and $\hat{x}=r_ms^{\ast\ast}\hat{y}$ imply that $x,\hat{x}$ are in $N_{r_ms^{\ast\ast}}\subseteqq N_{r_mt_{m}}$, then $\|\sum\limits_{i=0}^{l}(\theta(u)(i)-\theta(\hat{u})(i))y_i\|<2^{-j}$
\end{proof}

 By repeating apply claim 1 and claim 2 to define number sequences $b_0<b_1<b_2....$, $l_0<l_1<l_2.....$, finite sequences $s_0,s_1,s_2...$ and dense open sets $D^j_i\subseteqq \mathbb{Z}(\vec{\epsilon})\cap[-1,1]^{\mathbb{N}}$ $i,j\in \mathbb{N}$. Let $b_0=l_0=0$,suppose we have $b_j,l_j$ and $D_i^{u}$ for $u<j$. Using claim 1 for the $j$ with $k=b_j+1$ to get $l_{j+1}$, a finite sequence $s_j^{\ast}$ and a comeager set $D^j$. we can assume that $l_{j+1}>l_j$ and $D^j\subseteqq C$.

 Let $D^j_0\supseteqq D^j_1\supseteqq D^j_2.....$ be dense open sets of $\mathbb{Z}(\vec{\epsilon})\cap[-1,1]^{\mathbb{N}}$ s.t. $\bigcap^{\infty}_{i=0}D^j_i\subseteqq D^j$. Now apply claim 2 with $j$,$k=b_j+1+\textrm{len}(s^{\ast}_j)$, $l=l_{j+1}$ and $G=\bigcap^{j}_{j'=0}D^{j'}_j$ to obtain $s^{\ast\ast}_j$.

 Let $s_j=s_j^{\ast}s_j^{\ast\ast}$ and $b_{j+1}=b_j+\textrm{len}(s_j)+1$ let $C'$ be the set of all $u\in \mathbb{Z}(\vec{\epsilon})\cap[-1,1]^{\mathbb{N}}$ of the form $\langle a_0\rangle s_0 \langle a_1\rangle ...$. Easily we have $\forall u,\hat{u}\in C'$

 {\rm (1)}\quad if $u(b_i)=\hat{u}(b_i)$ for all $i\geq j+1$, then $\|\sum\limits_{i=l_{j+1}}^{\infty}(\theta(u)(i)-\theta(\hat{u})(i))y_i\|<2^{-j}$.

 {\rm (2)}\quad if $u(b_i)=\hat{u}(b_i)$ for all $i\leq j$, then  $\|\sum\limits_{i=0}^{l_{j+1}}(\theta(u)(i)-\theta(\hat{u})(i))y_i\|<2^{-j}$

 As $\{x_{b_n}\}$ does not dominate the $\{y_n\}$, it implies that there is a sequence $\{\delta_i\}$ such that $\sum\delta_ix_{b_i}$ converges but $\sum\delta_iy_i$ diverges. Here, as $\sum\delta_ix_{b_i}$ converges, $\lim\limits_{i\rightarrow \infty} \delta_i=0$. Then, we can assume that $|\delta_i|<1/2$. we can also assume that for any $i$, $|\delta_i|>\epsilon_{b_i}=2^{-b_i}$ by adding $\epsilon_{b_i}$ to the original $|\delta_i|$. Then we obtain that $\epsilon_{b_i}<|\delta_i|<1$

 Now define a function $g$ from $ \mathbb{Z}(\vec{\delta})\cap[-1,1]^{\mathbb{N}}$ to $ \mathbb{Z}(\vec{\epsilon})\cap[-1,1]^{\mathbb{N}}$ as follows:

 Firstly, we set that $g(x)$ is of the form $\langle a_0\rangle s_0 \langle a_1\rangle ...$. Then, we just need to define the value of $g(x)$ in $b_i$ as follows.

 $g(x)(b_i)=q_i\epsilon_{b_i}$ if $x(b_i)=p_i\delta_i$ such that $|q_i|$ is the max number  s.t. $|p_i\delta_i-q_i\epsilon_{b_i}|<2^{-b_i}$.

 The function is well-defined. As $|\delta_i|>\epsilon_{b_i}=2^{-b_i}$. if $\delta_i>0$, by the induction of $p\in \mathbb{N}$, we can easily prove that for all $p\in \mathbb{N}$, there is a $q\in \mathbb{N}$ such that $p\delta_i=q\epsilon_{b_i}+\mu$, where, $\mu <\epsilon_{b_i}$. For the case that $\delta_i<0$, the same method works.

 It is easy to check that $g(u)\in C'$ and we need to further check that $u-\hat{u}\in {\rm coef}(X,(x_{b_n}))$ iff $g(u)-g(\hat{u})\in {\rm coef}(X,(x_n))$. As $m_i=|g(u)(b_i)-g(\hat{u})(b_i)-(u(i)-\hat{u}(i))|<2^{-b_i+1}$, $\sum m_i$ converges. Consequently, $\sum(g(u)(b_i)-g(\hat{u})(b_i))x_{b_i}$ converges iff $\sum (u(i)-\hat{u}(i))x_{b_i}$ converges.

Then we have followings:
 $\sum(g(u)(i)-g(\hat{u})(i))x_i$ converges iff $\sum(g(u)(b_i)-g(\hat{u})(b_i))x_{b_i}$ converges iff  $\sum (u(i)-\hat{u}(i))x_{b_i}$ converges.

Now,define $\theta'\colon \mathbb{Z}(\vec{\delta})\cap[-1,1]^{\mathbb{N}}\rightarrow \mathbb{R}^\mathbb{N}$ by for all $j$ and all $m$ such that $l_j<m \leq l_{j+1}$, define

 $\theta'(u)(m)=\theta(g(e_j(u)))(m)$ where
 \[e_j(u)(i)=\left\{\begin{array}{ll}
u(j)&\textrm{if $i=j$}\\
0&\textrm{if $i\neq j$}
\end{array}\right.\]

In fact, we can take $\theta'(u)$ of the form $f_0(u(0))^{\smallfrown}f_1(u(1))^{\smallfrown}f_2(u(2))...$, where $f_j(a)=\theta(g(\langle 0,0,...0,a,0....\rangle))|(l_j,l_{j+1}]$

\emph{\textbf{Claim 3}.}
$\forall u,\hat{u}\in \mathbb{Z}(\vec{\delta})\cap[-1,1]^{\mathbb{N}}$, $u-\hat{u}\in {\rm coef}(X,(x_{b_n}))$ iff $\theta'(u)-\theta'(\hat{u})\in {\rm coef}(Y,(y_n))$.

\begin{proof}
(Claim 3.) we need to show $\theta'(u)-\theta(g(u))\in {\rm coef}(Y,(y_n))$

 Define \[e'_j(u)(i)=\left\{\begin{array}{ll}
u(j)&\textrm{if $i\leq j$}\\
0&\textrm{if $i>j$}
\end{array}\right.\]

by the claim 1,2 we have followings:

$\|\sum\limits_{i=0}^{l_{j+1}}(\theta(g(u))(i)-\theta(g(e_j'(u)))(i))y_i\|<2^{-j}$ and

$\|\sum\limits_{i=l_j+1}^{\infty}(\theta(g(e_j'(u)))(i)-\theta(g(e_j(u)))(i))y_i\|<2\cdot2^{-j}$

As the condition 2 of $\{y_n\}$ implies $\{y_n\}$ is unconditional, we have $\|\sum\limits_{i=l_j+1}^{l_{j+1}}(\theta(g(e_j'(u)))(i)-\theta(g(e_j(u)))(i))y_i\|<R\|\sum\limits_{i=0}^{l_{j+1}}(\theta(g(e_j'(u)))(i)-\theta(g(e_j(u)))(i))y_i\|$ for some $R$.

Therefore, we have $\|\sum\limits_{i=l_j+1}^{l_{j+1}}((\theta(g(u))(i)-\theta'(u)(i))y_i\|<(R+2)\cdot2^{-j}$.

Thus, $\forall \epsilon$, for sufficient large $j$ the followings hold.

$\|\sum\limits_{i=l_j+1}^{\infty}((\theta(g(u))(i)-\theta'(u)(i))y_i\|\leq \sum\limits_{n=j}^{\infty}\|\sum\limits_{i=l_j+1}^{l_{j+1}}((\theta(g(u))(i)-\theta'(u)(i))y_i\|<\epsilon$.

Thus, $\theta'(u)-\theta(g(u))\in {\rm coef}(Y,(y_n))$.
\end{proof}

\emph{\textbf{Claim 4}.}
There are $C\in \mathbb{R}^{+},N\in \omega$ such that $\forall j>N$, $\forall u\in \textrm{dom}(f_j)$ s.t.$|u|> \frac{1}{2}$, $\|\sum\limits_{i=l_j+1}^{l_{j+1}}(f_j(u)(i)-f_j(0)(i))y_i\|>C$.

\begin{proof}
(Claim 4.) Assume not, $\forall n$, $N_n=\max\{n,j_{n-1}\}$, $\exists j_n>N_n$ and $\exists u_n\in dom(f_{j_n})$ s.t. $|u_n|> \frac{1}{2}$ s.t.

$\|\sum\limits_{i=l_{j_n}+1}^{l_{j_n+1}}(f_{j_n}(u_n)(i)-f_{j_n}(0)(i))y_i\|<2^{-n}$

Take $\hat{u}=\vec{0}$ and $u$ s.t. \[u(i)=\left\{\begin{array}{ll}
u_n&\textrm{if $i=j_n$}\\
0&\textrm{otherwise}
\end{array}\right.\]
$u-\hat{u}=x\nrightarrow 0$. However, for sufficient large $m$, $\|\sum\limits_{i=l_{j_m+1}}^{\infty}(\theta'(u)(i)-\theta'(\hat{u})(i))y_i\|\leq \sum\limits_{n=m}^{\infty}\|\sum\limits_{i=l_{j_n+1}}^{l_{{j_n}+1}}(f_{j_n}(u_n)(i)-f_{j_n}(0)(i))y_i\|\leq \sum\limits_{n=m}^{\infty}2^{-n}\leq \epsilon$.

A contradiction.
\end{proof}

As $\sum\delta_ix_{b_i}$ converges, then, $\delta_j\rightarrow 0$. Thus, for $N$ occurs in the claim 4,  there is a $M\geq N$ such that $j>M$ implies that $|\delta_j|<1/2$. In the case that $j>M$, let $k_j\in \omega$ satisfies that $\gamma_j=k_j\delta_j$ such that $|\gamma_j|\in [\frac{1}{2},1]$. By claim 4, $\|\sum\limits_{i=l_j+1}^{l_{j+1}}(f_j(\gamma_j)(i)-f_j(0)(i))y_i\|>C$ for some $C$. Then
$\|\sum\limits_{i=l_j+1}^{l_{j+1}}(f_j(\delta_j)(i)-f_j(0)(i))y_i\|+....+\|\sum\limits_{i=l_j+1}^{l_{j+1}}(f_j(\gamma_j)(i)-f_j((k_i-1)\delta_j)(i))y_i\|>C$.

Therefore, there is some number $n_i\leq k_i$ s.t. $\|\sum\limits_{i=l_j+1}^{l_{j+1}}(f_j((n_j+1)\delta_j)(i)-f_j(n_j\delta_j)(i))y_i\|\geq C/k_j=C\delta_j/\gamma_j$.

Define a block basis of $\{y_n\}$ by $S_j=\sum\limits_{i=l_j+1}^{l_{j+1}}(f_j((n_j+1)\delta_j)(i)-f_j(n_j\delta_j)(i))y_i$. If $j>M$, then $\|S_j\|\geq C\delta_j/\gamma_j$.

set $u$ s.t.$u(i)=(n_i+1)\delta_i$ and $\hat{u}$ s.t. $\hat{u}(i)=n_i\delta_i$. Then $\sum (u(i)-\hat{u}(i))x_{b_i}=\sum\delta_ix_{b_i}$ converges.
Take $\{s_j=S_j/\|S_j\|\}$ to be a normalized block bases of $\{y_j\}$. Undoubtedly, $\sum S_j$ converges iff $\sum_{j=M}^{\infty}\|S_j\|s_j$ converges. Then, as $\{y_j\}$ are lower semi-homogeneous, $\sum_{j=M}^{\infty}\|S_j\|s_j$ converging implies that $\sum_{j=M}^{\infty}\|S_j\|y_j$ converges.

As $\|S_j\|\geq C\delta_j/\gamma_j \geq C\delta_j$ for $j>M$ and $\{y_j\}$ lower semi-homogeneous implying that it is unconditional, the following holds by applying proposition 2.6:
$\sum S_j$ converging implies that $\sum_{j=M}^{\infty}\delta_jy_j$ converges. Now, as $\sum_{j=M}^{\infty}\delta_jy_j$ diverges, $\sum (\theta'(u)(i)-\theta'(\hat{u})(i))y_i=\sum S_j$ diverges. It means that $\theta'$ is not a reduction.

A contradiction to claim 3.
\end{proof}

Using Theorem 1.2 we can also arrive a lot of interesting corollaries. Some known results about classical sequence Banach spaces, due to Dougherty and Hjorth, can be arrived by this theorem. The proof is easy if we notice the difference of two spaces and all natural basis of these space are subsequence equivalent and perfectly homogeneous, thus lower semi-homogeneous.

\begin{corollary}[\textup{Dougherty and Hjorth,} \cite{D-H},\cite{H}]
For classical sequences Banach spaces $l_p$, where $p \geq 1$ and $c_0$, we have $c_0 \nleq_B l_p$ for any $p \geq 1$ and $l_q \nleq_B l_p$ if $p<q$.
\end{corollary}

On the other hand, Tsirelson's space $T$ and its dual $T^\ast$, which is actually the original space constructed by Tsirelson (see \cite{C-S} and \cite{T}) serve to solve a well-known question in Banach space theory that whether there is a Banach space contains no isomorphic copies of $c_0$ and $l_p$ for $p \geq 1$. The natural analogous question, asked by Ding, that whether there is a equivalence relation of the form $E(X, (x_n))$ witnessing that  neither, for some $p \geq 1$, $\mathbb{R}^\mathbb{N}/l_p \leq E(X,(x_n))$ nor $\mathbb{R}^\mathbb{N}/c_0 \leq E(X,(x_n))$. In this section we only solve the case of $\mathbb{R}^\mathbb{N}/l_p$ for $p>1$ and $\mathbb{R}^\mathbb{N}/c_0$. This partly prove Theorem 1.3. For the case $l_1$, Please see the last section of this paper. The following proposition is essentially due to Casazza, Johnson and Tzafriri (see Lemma 4 in \cite{C-J-T} and Corollary 2.2 in \cite{C-S}).

\begin{proposition}[\textup{Casazza, Johnson and Tzafriri,} \cite{C-J-T}]
The natural basis $\{t_n\}$ in Tsirelson space $T$ is lower semi-homogeneous.
\end{proposition}

\begin{corollary}
For any $p> 1$, Neither $\mathbb{R}^\mathbb{N}/l_p$ nor $\mathbb{R}^\mathbb{N}/c_0$ Borel reducible to $E(T,(t_n))$.
\end{corollary}
\begin{proof}
By the proposition above, we know that the sequence of unit vectors $\{t_n\}$ is lower semi-homogeneous in $T$. Thus as the unit vector basis $\{e_n\}$ in both $l_p$ and $c_0$ are subsymmetric (actually symmetric), we just need to show that $\{e_{n}\}$ does not dominate the $\{t_n\}$. It is easily to see that the sequence $(\frac{1}{n})_{n=1}^{\infty}$ witnesses that $\sum\limits_{n=0}^{\infty}\frac{1}{n+1} t_n$ diverges in $T$, while $\sum\limits_{n=0}^{\infty}\frac{1}{n+1}e_n$ converging in both $l_p$ and $c_0$.
\end{proof}

Now, we give some corollaries concerning Schauder equivalence relations generated by Lorentz sequence spaces and Orlicz sequence spaces. It is easy to see that $E(d(w,p),(e_n))$ is just the equivalence relation $\mathbb{R}^\mathbb{N}/d(w,p)$ and $E(h_M,(e_n))$ is $\mathbb{R}^\mathbb{N}/h_M$, similarly.

Firstly, We know that $l_p$ is a proper subspace of $d(w,p)$ (see Proposition 4.e.3 in \cite{L-T}). Thus, by Theorem 1.1 and Theorem 1.2, we have the following corollary.

\begin{corollary}
For any $p \geq 1$, $\mathbb{R}^\mathbb{N}/l_p <_B \mathbb{R}^\mathbb{N}/d(w,p)$
\end{corollary}

For Orlicz spaces, we just study $h_M$ as unit vectors $\{e_n\}$ form a symmetric basis of it. If $M$ satisfies $\Delta^{'}$ condition, $\{e_n\}$ then in $h_M$ is lower semi-homogeneous.

Consider $u_j=\sum\limits^{p_{j+1}}_{n=p_j+1}a_ne_n$ is a normalized block basis of $\{e_n\}$. We can easily check that $\sum\limits^{p_{j+1}}_{i=p_j+1}M(|a_i|)=1$. For any scalar sequence $\{b_n\}\in \mathbb{R}^{\mathbb{N}}$, if $\sum\limits_{j=0}^{\infty}b_ju_j$ converges, then $\sum\limits_{j=0}^{\infty}\sum\limits^{p_{j+1}}_{i=p_j+1}M(|a_ib_j|)<\infty$. As $M$ satisfies $\Delta^{'}$ condition,
$$\sum\limits_{j=0}^{\infty}\sum\limits^{p_{j+1}}_{i=p_j+1}M(|a_ib_j|) \geq \sum\limits_{j=0}^{\infty}cM(|b_j|)\sum\limits^{p_{j+1}}_{i=p_j+1}M(|a_i|)=c\sum\limits_{j=0}^{\infty}M(|b_j|).$$
It means that $\sum\limits_{j=0}^{\infty}b_je_j$ converges. Thus, $\{e_n\}$ in $h_M$ is lower semi-homogeneous.

Thus, For Orlicz spaces $h_M$ and $h_N$, Theorem 1.2 has following special form:

\begin{corollary}
$M$ and $N$ are Orlicz functions. If two following conditions hold:

{\rm (1)}\quad $h_M\nsubseteq h_N$, and

{\rm (2)}\quad $N$ satisfies $\Delta^{'}$ conditions,

then $\mathbb{R}^\mathbb{N}/h_M \nleq_B \mathbb{R}^\mathbb{N}/h_N$.

\end{corollary}
We would like to provide more application of Theorem 1.1 and Theorem 1.2 in the next section to study the boundaries of the Schauder equivalence relations.

\section{Boundaries of Schauder equivalence relations}
In the last two sections, we will study some structural properties of the class of the Schauder equivalence relations. We denote $\mathcal{A}$ for the class, of all equivalence relations of the form $E(X,(x_n))$ and $\mathcal{A}_U$ the subset of $\mathcal{A}$ such that $\{x_n\}$ is unconditional. Firstly, we prove that $\mathcal{A}$ and $\mathcal{A}_U$ have boundaries to some extent.

For any Banach space $X$ with a normalized basic sequence $\{x_n\}$, it is easy to see that $l_1\subset {\rm coef}(X, (x_n)) \subset c_0$. It seems that $\mathbb{R}^\mathbb{N}/l_1$ and $\mathbb{R}^\mathbb{N}/c_0$ are special ones in $\mathcal{A}$. Indeed they are. We claim that $\mathbb{R}^\mathbb{N}/c_0$ and $\mathbb{R}^\mathbb{N}/l_1$ are minimal incomparable ones in $\mathcal{A}$ in the order of $\leq_B$ by proving the following proposition.

\begin{remark}
For the part of $l_1$. we can use Theorem 2.1. As we can check directly that every $E(X, (x_n))$ is turbulent, then $\mathbb{R}^\mathbb{N}/l_1$ is minimal in $\mathcal{A}$.  To be self-contained in this paper, we can use Theorem 1.2 to prove the minimality of $\mathbb{R}^\mathbb{N}/l_1$.
\end{remark}

\begin{proposition}
If $c_0$ (\emph{resp}. $l_1$) can not be embedded into $X$ with a basis $\{x_n\}$, then $E(X, (x_n))\nleq_B \mathbb{R}^\mathbb{N}/c_0$ (\emph{resp}. $\mathbb{R}^\mathbb{N}/l_1$).
\end{proposition}
\begin{proof}
Firstly, we can assume that $\{x_n\}$ is normalized. Then, we use Theorem 1.2 to prove the minimality of $\mathbb{R}^\mathbb{N}/l_1$. $l_1$ with $\{e_n\}$ is perfectly homogeneous, thus lower semi-homogeneous. Given ${\rm coef}(X,(x_n))$, for any subsequences $\{x_{k_n}\}$ of $\{x_n\}$, as $l_1$ can not being embedded into $X$, $\{e_n\}$ can not be equivalent to $\{x_{k_n}\}$. Thus, $l_1\subsetneqq {\rm coef}(X,(x_{k_n}))$. Now, Theorem 1.2 applies.

The case for $c_0$ is a little different. Given ${\rm coef}(X,(x_n))$, we will use the first three claims of the proof of Theorem 1.2 instead of itself. Then, similar to the case of $l_1$, we have for a subsequence $\{x_{b_n}\}$ of $\{x_n\}$, $ {\rm coef}(X,(x_{k_n})) \subsetneqq c_0$. It means that there is a sequence $(\delta_i)$ such that $\epsilon_{b_i}<|\delta_i|<1$ with $|\delta_i|\rightarrow 0$ but $\sum \delta_i x_{b_i}$ diverging. In the same way in Theorem 1.2, we thus can obtain a ``modular'' reduction $\theta'$ witnesses that $E(X,(x_{k_n}))\leq_A c_0$ with the form $\theta'(u)=f_0(u(0))^{\smallfrown}f_1(u(1))^{\smallfrown}f_2(u(2))...$.

As $\sum \delta_i x_{b_i}$ diverging, $a_n=\|f_n(\delta_n)-f_n(0)\|_{c_0}\nrightarrow 0$ holds. In fact we can assume that there is a $\epsilon_0$ such that $a_n=\|f_n(\delta_n)-f_n(0)\|_{c_0}>\epsilon_0$ as we can always choose a subsequence $\{a_{n_k}\}$ of $\{a_n\}$ such that $a_{n_k}>\epsilon_0$ as $a_n\nrightarrow 0$ for some $\epsilon_0$ and use $\{a_{n_k}\}$ to replace $\{a_n\}$. In this case, there is no subsequence $\{a_{n_k}\}$ of $\{a_n\}$ such that $a_{n_k}\rightarrow 0$. However, As $|\delta_i|\rightarrow 0$, we can choose a subsequence $\{\delta_{p_i}\}$ of $\{\delta_{i}\}$  such that $\{\delta_{p_i}\}\in l_1$ forces that $\sum \delta_{p_i} x_{b_{p_i}}$ converges. Define $\delta'$ as follows.
\[\delta'_n=\left\{\begin{array}{ll}
\delta_{p_k}&\textrm{if $n=p_k$}\\
0&\textrm{otherwise}
\end{array}\right.\].

In this case, as $\theta'$ is a reduction, $a_{p_n}=\|f_{p_n}(\delta_{p_n})-f_{p_n}(0)\|_{c_0}\rightarrow 0$.
A contradiction.
\end{proof}

This proposition, together with Theorem 1.1 and the result of same type about $l_1$. We can show the minimality of $\mathbb{R}^\mathbb{N}/c_0$ and $\mathbb{R}^\mathbb{N}/l_1$.

\begin{theorem}
If $E(X, (x_n))\leq_B \mathbb{R}^\mathbb{N}/c_0$ ( \emph{resp}. $\mathbb{R}^\mathbb{N}/l_1$ ) , then $E(X, (x_n))\thicksim_B \mathbb{R}^\mathbb{N}/c_0$ (\emph{resp}. $\mathbb{R}^\mathbb{N}/l_1$ ).
\end{theorem}
\begin{proof}
We just show the case of $c_0$, as the case of $l_1$ shares the same proof. If $E(X, (x_n))\leq_B \mathbb{R}^\mathbb{N}/c_0$, by the proposition above, we must have $c_0$ can be embedded in $X$. By Theorem 1.1, we have that $\mathbb{R}^\mathbb{N}/c_0\leq_B E(X, (x_n))$. Thus $E(X, (x_n))\thicksim_B \mathbb{R}^\mathbb{N}/c_0$.
\end{proof}

For the upper boundaries of $\mathcal{A}$ and $\mathcal{A}_U$. we appeal to the universal separable Banach spaces $U_1$ and $U_2$ constructed by Pelczynski\cite{P} (see also \cite{S}). $U_1$ has an unconditional basis $\{u_i\}$ such that every unconditional basic sequence (in an arbitrary separable Banach space) is equivalent to a subsequence of $\{u_i\}$. For $U_2$, similarly, has a Schauder basis $\{v_n\}$ such that every basic sequence is equivalent to one of its subsequence. Thus, we can easily see the following theorem.

\begin{theorem}
For any  $E(X, (x_n))$ in $\mathcal{A}_U$, $E(X, (x_n))\leq_B E(U_1, (u_n))$ and for any $E(X, (x_n))$ in $\mathcal{A}$, $E(X, (x_n))\leq_B E(U_2, (v_n))$.
\end{theorem}

\section{Bases of Schauder equivalence relations}
In this section, Farah's conclusion of Tsirelson ideals are used to prove all kinds of non-reducibility concerning the Tsirelson space. We mainly prove Proposition 1.3 and Theorem 1.4. Similar to the argument of Farah, Theorem 1.4 leads Corollary 1.5 naturally. For any unconditional normalized basic sequence $\{x_n\}$ in $X$, we denote $I_{(x_n),f}^X$  for the ideal  $\emph{I}=\{A \in P(\mathbb{N}) \colon \sum\limits_{n\in A} f(n)x_n\ \textrm{converges}\}$. For this kind of ideals, we can define following submeasure:

\[\varphi(A)=\left\{\begin{array}{ll}
\||\sum\limits_{n\in A}f(n)x_n\||&\textrm{if $\sum\limits_{n\in A}f(n)x_n$ converges}\\
\sup\limits_m\||\sum\limits_{n\in A\bigcap [0,m)}f(n)x_n\||&\textrm{otherwise},
\end{array}\right.\]

where $\||\cdot\||$ is the norm which is equivalent to the original norm $\|\cdot\|$ on $[x_n]^{\infty}_{n=0}$ but makes $\{x_n\}$ monotone. It is easy to see that $\varphi(A)$ is a LSC submeasure and $I_{(x_n),f}^X=Exh_{\varphi}$. Thus, by Theorem 2.8 and Remark 2.9 following it, $I_{(x_n),f}^X$ is turbulent if and only if $f(n)\rightarrow 0$. The following lemma is fundamental.

\begin{lemma}
For $X$ is a Banach space having no subspaces isomorphic to $c_0$. $Y$ is also a Banach space. Assume $E(X, (x_n))\leq_B E(Y, (y_n))$ with $\{x_n\}$ and $\{y_n\}$ being unconditional and monotone, respectively, then there are functions $f,g \colon \mathbb{N} \rightarrow \mathbb{R}^+$ with $f(n),g(n)\rightarrow 0$, subsequence $\{x_{b_n}\}$ of $\{x_{n}\}$, and a normalized block basis $\{s_j\}$ of $\{y_n\}$ such that $I_{(x_{b_n}),f}^X= I_{(s_{n}),g}^Y$.
\end{lemma}

\begin{proof}
Using the proof of Theorem 1.2 (claim 1-3), After finding a subsequence $\{x_{b_n}\}$ of $\{x_n\}$, we need to construct the ``modular'' reduction. Now, we do not need the difference between the ${\rm coef}(X, (x_{b_n}))$ and ${\rm coef}(Y, (y_n))$. Thus we follow the original step of Dougherty and Hjorth \cite{D-H} by taking $\delta_i=\epsilon_{b_i}$ and $g(u)(b_i)= u(i)$. We thus can obtain a ``modular'' reduction $\theta'$ witnesses that $E(X,(x_{k_n}))\leq_A E(Y, (y_n))$ with the form $\theta'(u)=T_1^{\smallfrown}(u(1))^{\smallfrown}T_2(u(2))^{\smallfrown}T_3(u(3))...$.  we can assume that $\theta'(\vec{0})=\vec{0}$ by define another reduction $\theta''(a)=\theta'(a)-\theta'(0)$. In this case, for any $n$, $T_n(0)=\vec{0}$.

As $X$ does not contain $c_0$, there is a function  $f \colon \mathbb{N} \rightarrow \mathbb{R}^+$ such that $\sum f(n)x_{b_n}$ diveges but $f(n)\rightarrow 0$ (see \cite[Propostion 2.e.4]{L-T} and the remark following it). Furthermore, we can asumme that for each $n$, $f(n)=k_n\delta_n$ for some $k_n\in \mathbb{N}$.

For any such a function $f$, we can define $\phi$ from $2^\mathbb{N}$ to  $\mathbb{Z}(\vec{\delta})\cap[-1,1]^{\mathbb{N}}$ by $\phi(a)(i)=f(i)\cdot a(i)$. Thus by combining $\theta'$ and $\phi$, we can define a reduction $\varphi$ witness that $2^\mathbb{N}/I_{(x_{b_n}),f}^X \leq_A E(Y, (y_n))$  by:
$$\varphi(a)=T_1(f(1)a(1))^{\smallfrown} T_2(f(2)a(2))^{\smallfrown}T_3(f(3)a(3))...$$

As $T_n(0)=\vec{0}$ holds, we in fact have following formula:
$$\varphi(a)=a(1)\cdot T_1(f(1))^{\smallfrown} a(2)\cdot T_2(f(2))^{\smallfrown} a(3)\cdot T_3(f(3))...$$

Take the block basis $S_j=\sum\limits_{i=l_j+1}^{l_{j+1}}(T_j(f(j))(i))y_i$ and $s_j=S_j/\|S_j\|$. We define $g(j)=\|S_j\|$.
We can easily check following holds for any $a,b\in 2^\mathbb{N}$:

$a\triangle b\in I_{(x_{b_n}),f}^X $ iff $\sum\limits_{a(i)\neq b(i)} f(i)x_{b_i}$ converges iff $ \phi(a)-\phi(b)\in {\rm coef}(X, (x_{b_n}))$ iff $\theta'(\phi(a))-\theta'(\phi(b))\in {\rm coef}(Y, (y_{n}))$ iff $\sum\limits_{j=1}^{\infty}|a(j)-b(j)|\cdot\|S_j\|s_j$ converges iff $a\triangle b\in I_{(s_n),g}^Y $

Thus $I_{(x_{b_n}),f}^X= I_{(s_{n}),g}^Y$. As $f(n)\rightarrow 0$, $I_{(x_{b_n}),f}^X$ is turbulent. Then $g(n)\rightarrow 0$.
\end{proof}

The following lemma concerning Tsirelson space is due to Casazza, Johnson and Tzafriri.

\begin{proposition}
 Let $\{y_j\}$ in $T$ of the form $y_j=\sum\limits^{p_{j+1}}_{p_j+1}a_nt_n$, with $\{a_n\}$ scalars, is a normalized block basic sequence of $\{t_n\}$, Then for every choice of natural numbers $p_j<k_j\leq p_{j+1}$, and every sequence of scalars $\{b_n\}$, we have:
 $$\frac{1}{3}\|\sum\limits_j b_jt_{k_j}\|\leq \|\sum\limits_j b_jy_j\|\leq 18\|\sum\limits_j b_jt_{k_j}\|$$.
\end{proposition}

It is worthy noting that all theorems above also hold in $T_{\alpha}$. see notes and remarks in X.A in \cite{C-S}. Due to Farah, start from Tsirelson space $T_{\alpha}$, Tsirelson ideals can be defined:
$$\mathscr{T}_{f,h,\alpha}=\textrm{Exh}(\tau_{f,h,\alpha})=I_{(t_{h(n)}),f}^T.$$.
When $\alpha$ is $1/2$, we write $\mathscr{T}_{f,h,1/2}$ to be $\mathscr{T}_{f,h}$. Farah studied this type of ideals thoroughly to refute a conjecture of Mazur and Kechris. Furthermore, He proved that every basis of turbulent orbit equivalence relations induced by continuous Polish group actions on Polish spaces is of size continuum. Here we only mentions his two propositions. For more details, please see \cite{F2} and \cite{F1}.

\begin{proposition}[\textup{Farah} \cite{F1}]
Each ideal $\mathscr{T}_{f,h}$ is different from $\mathscr{I}_{1/n}$.
\end{proposition}

Based on the lemma and the propositions above, we can proved that $\mathbb{R}^\mathbb{N}/l_1 \nleq_B E(T, (t_n))$. With Corollary 3.14, we finished the proof of Proposition 1.3.

\begin{proof}
(Proposition 1.3) If $\mathbb{R}^\mathbb{N}/l_1 \leq_B E(T, (t_n))$, by the lemma above, we can find a $f$, with $f(n)\rightarrow 0$, such that $I_{1/n}= I_{(e_{n}),f}^{l_1}=I_{(s_{n}),g}^T$ for some $g$ with $g(j)\rightarrow 0$, and a normalized block basis $\{s_j=\sum\limits_{i=l_j+1}^{l_{j+1}}a_it_i\}$.  From Proposition 5.2,  we know that $\{s_j\}$ is equivalent to $\{t_{k_j}\}$ for any $l_j<k_j \leq l_{j+1}$. Thus, $I_{1/n}= I_{(t_{k_n}),g}^T= \mathscr{T}_{k,g}$. A contradicition to Proposition 5.3.
\end{proof}

Now we are ready to prove Theorem 1.4. The following proposition is also due to Farah.
\begin{theorem}[\textup{Farah} \cite{F2}]
If both $\mathscr{T}_{f_1,h_1,\alpha}$ and $\mathscr{T}_{f_2,h_2,\beta}$, with $\alpha \neq \beta$, are turbulent, then they are different.
\end{theorem}

\begin{proof}
(Theorem 1.4) If there is a $E(X,(x_n))$ Borel reducible to $E(T_{\alpha},(t_n^{\alpha}))$ and $E(T_{\beta},(t_n^{\beta}))$ with $\alpha \neq \beta$. We know that $X$ does not contain $c_0$, as $\mathbb{R}^\mathbb{N}/c_0 \nleq_B E(T_{\alpha},(t_n^{\alpha}))$ Then as in the proof of lemma 5.1, we can find a subsequence $\{x_{b_n}\}$ of $\{x_{n}\}$ and a reduction $\theta_1$ of the ``modular'' form $\theta_1(u)=T_1^{\smallfrown}(u(1))^{\smallfrown}T_2(u(2))^{\smallfrown}T_3(u(3))...$, witnesses that $E(X,(x_{b_n})) \leq_A E(T_{\alpha},(t_n^{\alpha}))$. From $\{x_{b_n}\}$  we can repeat this steps to find one of its subsequence $\{x_{b_{d_n}}\}$ and a reduction $\theta_2$ of the ``modular'' form witnessing that  $E(X,(x_{b_{d_n}})) \leq_A E(T_{\beta},(t_n^{\beta}))$. Using the lemma 5.1 and the proof above, we have for some $f$ there is a $g_1$ and $k_1$ such that $I_{(x_{b_{d_n}}),f}^X= \mathscr{T}_{k_1,g_1,\beta}$, which is turbulent.

In addition, consider the domain of $\theta_1$ in coordinate $b_{d_n}$, we can construct a reduction $\theta'_1$ from $2^\mathbb{N}$ as follows:
$$\theta'_1(a)=\vec{0}^{\frown}\vec{0}^{\frown}\ldots a(1)\cdot T_{d_1}(f(1))^{\frown}\vec{0}^{\frown}\ldots a(2)\cdot T_{d_2}(f(2))^{\frown}\ldots.$$
In fact, for any $a,b\in 2^N$, we can check that
$a\triangle b\in I_{(x_{b_{d_n}}),f}^X $ iff $\sum\limits_{a(i)\neq b(i)} f(i)x_{b_{d_i}}$ converges iff $ \phi(a)-\phi(b)\in {\rm coef}(X, (x_{b_{d_n}}))$ iff $\theta'_1(a)-\theta'_1(b)\in {\rm coef}(T_{\alpha},(t_n^{\alpha}))$ iff $\sum\limits_{j=1}^{\infty}|a(j)-b(j)|\cdot\|S_{j}\|s_{j}$ converges iff $a\triangle b\in I_{(s_n),g_2}^{T_{\alpha}} $, where $S_j=\sum\limits_{i=l_{d_j}+1}^{l_{d_j+1}}(T_{d_j}(f(j))(i))t_i^{\alpha}$, $s_j=S_j/\|S_j\|$ and $g_2(j)=\|S_j\|$.

Similar to the proof of Proposition 1.3, we can get any subsequences $\{t_{k_j}\}$, with $l_{d_j}<k_j \leq l_{d_j}+1$, equivalent to $\{s_j\}$. Choose such a subsequence and define $k_2(j)=k_j$. Thus, we have $I_{(x_{b_{d_n}}),f}^X= \mathscr{T}_{k_2,g_2,\alpha}$ and then $\mathscr{T}_{k_1,g_1,\beta}=\mathscr{T}_{k_2,g_2,\alpha}$. Both of them are turbulent.

A contradiction to theorem 5.4.

\end{proof}

Then we are ready to prove Corollary 1.5. Comparing to Farah's Theorem, as every $E(X,(x_n))$ being turbulent, Corollary 1.5 shows that the same argument also holds for a subclass of the turbulence equivalence relations induced by continuous actions. In fact, his proof also holds in our setting for Corollary 1.5. See the proof of Theorem 1.2 in \cite{F1}. However, to be self-contained, we would like to provide the proof here.

\begin{proof}
(Corollary 1.5) As there are only continuum many Borel equivalence relations, it suffices to prove that if $E(X_{\xi}, (x^{\xi}_n))$, where $\xi<\lambda<2^{\omega}$, are equivalence relations in $\mathcal{A}$, then there is some equivalence relation $E$ in $\mathcal{A}$ such that $E(X_{\xi}, (x^{\xi}_n)) \nleq_B E$ for all $\xi<\lambda$. Based on Theorem 1.4, we know that for every $\xi$, there is at most one $\alpha_{\xi}$ such that $E(X_{\xi}, (x^{\xi}_n)) \leq_B E(T_{\alpha_{\xi}},(t^{\alpha_{\xi}}_n))$. Fix a $\alpha$,which is different from all $\alpha_{\xi}$. Then, take the equivalence relation $E$ to be $E(T_{\alpha},(t^{\alpha}_n))$.
\end{proof}

\section{Acknowledgement}
This paper is the main part of my master thesis in Nankai University. I would like to thank my supervisor Longyun Ding for his guidance and I also want to thank my classmates Zhi Yin, Minggang Yu and Yukun Zhang for the inspiring discussions.

\end{document}